\documentclass[12pt]{amsart}
\usepackage{amsmath,amsthm}
\usepackage{amssymb}
\usepackage{amsfonts}
\usepackage{amscd}

\newtheorem{thm}{Theorem}[section]
\newtheorem{theorem}[thm]{Theorem}

\newtheorem{corollary}[thm]{Corollary}
\newtheorem{lemma}[thm]{Lemma}

\newtheorem{proposition}[thm]{Proposition}

\theoremstyle{definition}

\newtheorem{remark}[thm]{Remark}

\newcommand{\N} {\mathbf{N}}
\newcommand{\R} {\mathbb{R}}
\newcommand{\Z} {\mathbf{Z}}

\newcommand{\C} {\mathbb{C}}

\newcommand{\cR}{\mathcal{R}}

\newcommand{\cH}{\mathcal {H}}

\newcommand{\cM}{{\mathcal M}}

\newcommand{\cP}{\mathcal{P}}

\newcommand{\cU}{\mathcal{U}}

\newcommand{{\cHol}}{{\mathcal{{H}}}}

\newcommand{{\cLf}}{\mathcal{L}_{\mathrm{f}}}
\newcommand{{\cLfo}}{\wt{\cLf}} %\mathcal{L}_{0,\mathrm{f}}}

\newcommand{\al}{\alpha}
\newcommand{\be}{\beta}
\newcommand{\ga}{\gamma}

\newcommand{\la}{\lambda}

\newcommand{\fgc}{{{\mathfrak g}_\C}}
\newcommand{\fkc}{{\mathfrak k}_\C}
\newcommand{\fhc}{{\mathfrak h}_\C}
\newcommand{\fpc}{{\mathfrak p}_\C}
\newcommand{\fg}{\mathfrak g}
\newcommand{\fgR}{\mathfrak{g}}
\newcommand{\fkR}{\mathfrak{k}}
\newcommand{\fpR}{\mathfrak{p}}

\newcommand{\fc}{\mathfrak c}

\newcommand{\fh}{\mathfrak h}
\newcommand{\fk}{\mathfrak k}

\newcommand{\fsl}{\mathfrak{sl}}

\newcommand{\bbar}{\overline b}

\newcommand{\rSL}{\mathrm{SL}}

\newcommand{\rOsp}{{\mathrm{Osp}}}
\newcommand{\rosp}{{\mathrm{osp}}}

\newcommand{\ad}{\mathrm{ad}}
\newcommand{\Lie}{\mathrm{Lie}}

\newcommand{\rk}{{\mathrm{rk}}}

\usepackage{color}

\newcommand{\lra} {\longrightarrow}
\newcommand{\beq} {\begin{equation}}
\newcommand{\eeq} {\end{equation}}
\newcommand{\xibar} {\overline{\xi}}

\begin{document}

\centerline{\Large\bf Unitary Harish-Chandra representations} 

\bigskip

\centerline{\Large \bf of real supergroups}

\bigskip

\centerline{C. Carmeli$^\natural$, R. Fioresi$^\flat$, V. S. Varadarajan $^\star$ }

\bigskip
\centerline{\it $^\natural$ DIME, 
Universit\`a di Genova, Genova, Italy}
\centerline{{\footnotesize e-mail: carmeli@dime.unige.it}}

\medskip
%\centerline{\it  $^\flat$ Dipartimento di Matematica, Universit\`{a} di
%Bologna }
% \centerline{\it Piazza di Porta S. Donato, 5. 40126 Bologna, Italy.}
\centerline{\it $\flat$ Universit\`{a} di Bologna, Bologna, Italy }
\centerline{{\footnotesize e-mail: rita.fioresi@unibo.it}}

\medskip
\centerline{\it $^\star$
%University of California, Los Angeles, 
UCLA, Los Angeles, CA 90095-1555, USA} %\comment{R. ADD ABSTRACT}
%\centerline{{\footnotesize e-mail: vsv@math.ucla.edu}}

\bigskip
\small
{\bf Abstract}. We give conditions for unitarizability of
Harish-Chandra super modules for Lie supergroups and superalgebras.

\normalsize
\section{Introduction}
\label{intro-sec}
Let $\fg$ be a real Lie superalgebra. %, $\fg_0$ reductive. 
It is natural to ask how to define the concept of 
(infinitesimal) \textit{unitarity or
unitarizability} for a super module $V$ for $\fg$ and how to
obtain, starting from $V$, a unitary module for $G$ a Lie supergroup 
with $\fg=\Lie(G)$, $G_0$ simply connected.
Let $\gamma$ be the representation
of $\fg$ in $V$. 
We say $\gamma$ is unitary if
$V$ is equipped with an hermitian product in which $V_0$ and $V_1$ are
orthogonal, and the following conditions are met 
(see  \cite{CCTV, jakobsen} and also \cite{neeb, salmasian}):\\
{\bf (U1)} For all $Z \in \fg_0$, $i\gamma(Z)$ is symmetric on $V$.\\
{\bf (U2)} For all $X \in \fg_1$, $\rho(X) := e^{-i\pi/4}\gamma(X)$ is
symmetric on $V$.

These are not enough in general to define in the completion $\cH$ of $V$ a 
unitary representation of a Lie supergroup $G$, with $\fg=\Lie(G)$,
whose infinitesimal form on $V$ is $\gamma$. Indeed,
as was remarked in Nelson \cite{nelson}, this is already not enough in the 
classical setting, that is when $\fg_1=0$. 
In general, we need an additional condition: \\
{\bf (U3)} There is an even unitary representation $\pi_0$ of $G_0$, 
the simply connected group  defined by $\fg_0$, on the completion 
$\cH$ of $V$ such that  $d\pi_0 (Z)$ is defined on
$V$ for all  $Z\in\fg_0$ and coincides with $\gamma(Z)$ on $V$; 
in the usual notation $V \subset \cH$ and
$\gamma(Z) \prec d\pi_0(Z)$, $Z \in \fg_0$
(see \cite{rs} Ch. 8 for definitions and notation).

We recall here that $d\pi_0(Z)$ is the unique self adjoint 
operator on $\cH$ such that $\pi_0(exp(tZ)) = e^{itd\pi_0(Z)}$. Then, 
Proposition 3 of \cite{CCTV} leads to the
following theorem.

\begin{theorem} \label{theorem1} 
Let $V$ be a module for a real Lie superalgebra $\fg$, %$\fg_0$ reductive, 
via the representation $\gamma$
such that conditions (U1)-(U3) are satisfied. Suppose that 
$V \subset C^\omega(\pi_0)$.
Then each $\rho(X)$ ($X \in \fg_1$) is essentially 
self-adjoint on $V$ with $C^\infty(\pi_0) \subset D(\overline{\rho(X)})$, 
and there is a unique unitary representation $(\pi_0, \rho, \cH)$ of the 
Lie supergroup $(G_0, \fg)$ in $\cH$
such that $\rho(X)$ is the restriction to  $C^\infty(\pi_0)$ of 
$\overline{\rho(X)}$ for all $X \in \fg_1$.
\end{theorem}

The shortcoming of this theorem is that it assumes the existence of $\pi_0$. 
As we shall see, when $\fg_0$ is reductive and the
modules are Harish-Chandra modules of $(\fg, \fk)$-type, 
then we can dispense with (U3) entirely. 
As a notational convention, when
we say that some module is a Harish-Chandra module, 
we assume it is already of $(\fg,\fk)$
type, i.e., each vector in the module lies in a direct sum of a finite number
of irreducible $\fk$-modules, with $\fk$ the maximal compact 
subalgebra in $\fg$ (see \cite{cfv}). 
We also recall that $\fk$ is not always 
semisimple and so not all of its finite dimensional modules are 
completely reducible;
the condition for complete reducibility is that each element of the center
of $\fk$ acts semisimply in the representation space.
We also shall use freely the super Harish-Chandra pairs (SHCP) terminology;
for all the notation and preliminaries on supergeometry refer to 
\cite{ccf, vsv2} Ch. 4 and 7, besides the classical references 
\cite{BerLeites, ma, Kostant, koszul}.

\medskip
We shall prove the following.

\begin{theorem}\label{main0}
Let $\fg$ be a real Lie superalgebra with $\fg_{0}$ reductive acting
via $\gamma$ on a complex vector superspace $V$. Assume:\\
%\item 
(U1) For all $Z \in \fg_0$, $i\gamma(Z)$ is  symmetric on $V$.\\
%\item 
(U2) For all $X \in \fg_1$, $\rho(X) := e^{-i\pi/4}\gamma(X)$ is 
 symmetric on $V$.\\
Let $G_0$ 
be the simply connected Lie group defined by $\fg_0$ and
let $G$ be the supergroup whose SHCP is $(G_{0}, \fg)$. 
If $V$ is finitely generated, then there is a unique unitary representation of
the Lie supergroup $G$ on the completion $\cH$ of $V$, 
say $(\pi_0, \rho, \cH)$ such that $V \subset C^\omega(\pi_0)$.
\end{theorem}

\medskip
We turn then to infinitesimal unitarity.
Let $\fgc$ the complexification of $\fg$, $\fgR=\fkR \oplus \fpR$
the Cartan decomposition and assume $\fg_\C$ contragredient.
Assume $\fg$ is \textit{equal rank}, that is
$\rk(\fk)=\rk(\fg)$ and that $\fkc$ has a
non trivial center (see \cite{cf1, cfv}). 
Fix $\fhc$ a Cartan subalgebra of $\fkc$ and $\fgc$.
Let $\Delta$ be the root system of $\fgc$, $\fgc=\fhc \oplus \sum_{\al \in \Delta}
\fg_\al$ the root space decomposition. 
The equal rank condition allows
us to decompose $\fkc$, $\fpc$ into root spaces; we say that
a root $\al$ is compact (non compact) if $\fg_\al \subset \fkc$ ($\fg_\al \subset \fpc$).
Let $\beta:Z(\fgc)\to S(\fh_\C)^W$
denote the Harish-Chandra homomorphism (see \cite{musson, ka4}).
We prove the following result.

\begin{theorem}\label{main1}
Let $\la \in \fh_\C^*$ and let $\pi_\lambda$ be the irreducible
highest weight representation of highest weight $\la$. Then 
$\pi_\la$ is unitary if and only if %$(-1)^{|a|}
$(-i)^{|a|}\beta(a^*a)(\la)>0$ 
for all $a \in \cU(\fgc)$.
In particular it is necessary that $\lambda(H_\al) \geq 0$
for $\al$ compact and $\lambda(H_\al) \leq 0$ for
$\al$ non compact even roots.
\end{theorem}

In the end we give an explicit example regarding $\fg_\C=\rosp_\C(1|2)$
and its real form $\rosp_\R(1|2)$ (\cite{ccf} Appendix A) proving the following.

\begin{theorem} \label{main-ex}
%Let $t \not\in \N=\{0,1,2\dots\}$.
Let $V_t$ be the universal (Verma) $\rosp_\C(1|2)$ module of highest weight $t$.
\begin{enumerate}
\item Then $V_t$ is irreducible and it is a unitary module for $\rosp_\R(1|2)$
if and only if $t$ is real and negative.
\item All unitary representation of
the real Lie supergroup $\rOsp_\R(1|2)$ $=$ $(\rSL_2(\R),$ $\rosp(1|2))$ 
are given on the
completion $\cH$ of $V_t$, 
and are such that $V_t \subset C^\omega(\pi_0)$, $\pi_0$ unitary 
representation of
$\rSL_2(\R)$ in $\cH$ integrating $(V_t)_0$.
\end{enumerate}
\end{theorem}

\bigskip
{\sl Acknowledgements}. R.F. and C.C. wish to thank the UCLA Dept. of
Mathematics for the warm hospitality during the realization of part 
of this work. R.F. research was funded by EU grant GHAIA 777822.

\section{$(\fg, \fk)$-supermodules and their unitarity} 
\label{unit-sec}

\subsection{Harish-Chandra modules for reductive Lie algebras}  
\label{hc-ordinary}

Let $\fg$ be a real reductive
Lie algebra. Then $\fg$ = $\fg'\oplus \fc$ where $\fc$ is the center 
of $\fg$ and $\fg' = [\fg, \fg]$ is
semisimple. Let $\fk \subset\fg'$  be a maximal subalgebra of compact type, 
which means that it is the set of fixed points of a Cartan involution 
of $\fg'$. Let $V$ be a ($\fg'$, $\fk$)-module. 
We recall that this means that $V$ is a $\fg'$-module which,
as a $\fk$-module, is a direct sum of finite dimensional irreducible 
$\fk$-modules.
Recall also that $\fk$ is reductive in $\fg$. Note that if $V$ is irreducible, 
then $\fc$ acts through an additive character on $V$ and so $V$ is 
an irreducible ($\fg$,$\fk$)-module.
This allows a reduction to the case when $\fg$ is itself semisimple.
One knows from Harish-Chandra's work (slightly modified to include the
reductive case) that if $V$ is irreducible, or more generally, is finitely 
generated as a $\cU(\fg')$-module on which $\fc$ acts semi simply 
through a finite number of additive characters, then the isotypical subspaces 
$V_\theta$ are all
finite dimensional, where $\theta$ runs through the set 
$\hat{\fk}$ of equivalence classes
of irreducible finite dimensional representations of $\fk$ 
($\fk$ is not in general semisimple, see \cite{vsv3} for details).
A basic question in the theory of $(\fg,\fk)$-modules is whether such a module
is the module of $\fk$-finite vectors of a Hilbert space representation (not
necessarily unitary) of the simply connected group $G$ defined by $\fg$. In
his paper \cite{hc-banach}, Harish-Chandra
proved this for irreducible $(\fg,\fk)$-modules which satisfy a certain 
condition.
He later verified %(see \cite{hc-iv}) 
that this condition is
satisfied for highest weight $(\fg,\fk)$-modules and so all such modules can be
realized as the $\fk$-finite vectors of Hilbert space representations of $G$. 
This is actually sufficient for our purposes. However, it is possible to
remove the special condition imposed by Harish-Chandra in his Theorem
4 in \cite{hc-banach}. The general result is as follows
(see \cite{wallach}, Ch. 8).

\begin{theorem} \label{Theorem1Wallach}
Any $(\fg,\fk)$-module
$V$ which is a direct sum of a finite number of irreducible submodules is
identifiable as the module of $\fk$-finite vectors of a Hilbert space 
(not necessarily unitary) representation $\pi$ of $G$. Moreover 
$V \subset C^\omega(\pi)$.
\end{theorem}

As mentioned earlier, when we deal with irreducible highest weight 
Harish-Chandra modules, the above general result is not needed, and 
Harish-Chandra already proves the above theorem for these modules.
%(see \cite{hc-iv}, Theorem 4).
The question arises if an irreducible $(\fg,\fk)$-module $V$, 
which is infinitesimally unitary, is the module of $\fk$-finite vectors 
of an irreducible unitary representation of $G$. In \cite{hc-banach} 
Harish-Chandra proves that such a unitary representation exists 
and is unique up to
equivalence provided $V$ is the module of $\fk$-finite vectors of a 
Banach space
representation of $G$ (Theorem 9, \cite{hc-banach}). In view of the 
above remarks
and results, we can now state the following theorem.

\begin{theorem}
\label{Theorem2} %(unitarity). 
Let $V$ be an irreducible $(\fg,\fk)$-module defined by
the representation $\gamma$ of $\cU(\fg)$, which is unitary 
in the sense that there is a
hermitian product $( , )$ on $V$ such that $i\gamma(X)$ is symmetric 
for all $X\in \fg$. Then,
there is a unitary representation of $G$ (unique up to unitary equivalence) 
in the completion
$\cH$ of $V$ with respect to the norm defined by the hermitian 
product, such that $V$
is the module of $\fk$-finite vectors in $\cH$.
\end{theorem}

\subsection{Unitarity of super modules} 
We shall now present a proof that for
a Lie superalgebra $\fg$ with $\fg_0$ reductive, conditions (U1) and (U2) of 
Sec. \ref{intro-sec} are enough to
guarantee the existence of a unitary representation of the Lie supergroup.
Let $\fg$ be a Lie superalgebra with $\fg_0$ reductive. Write, as in Subsec. 
\ref{hc-ordinary}, 
$\fg_0$ = $\fg_0'\oplus \fc_0$. $\fk_0$  is the
subalgebra of $\fg_0'$  fixed by a Cartan involution.

\begin{lemma}
\label{Lemma1} 
Let $V$ be a $(\fg_0, \fk_0)$-module. If $V$ admits a hermitian product which
is $\fk_0$-invariant, namely, elements of $\fk_0$  
are skew symmetric with respect to
it, then the isotypical subspaces $V_\theta$ are mutually orthogonal.
\end{lemma}

\begin{proof}
Let $V_1$, $V_2$ be two irreducible $\fk_0$ -stable finite 
dimensional subspaces such that they carry inequivalent 
representations of $\fk_0$. We want to prove
that $V_1 \perp V_2$. Let $W = V_1 \oplus V_2$. 
%(the sum is direct of course).
Let $P$ be
the orthogonal projection $V_1\lra V_2$ in $V$. 
We claim that $P$ is a $\fk_0$-map.
%Proof of claim.
Let $u \in  V_1$, write $u = x + y$ where 
$x \in V_2$, $y \in W$, $y \perp V_2$,
or $y \in V_2^\perp \cap W$. Then $Pu = x$ and for $X \in \fk_0$, 
$XP u = Xx$. On the
other hand, $Xu = Xx + Xy$ and we know that $Xx \in V_2$, 
$Xy \in W \cap V_2^\perp$.
Hence $P Xu = Xx = XP u$, proving the claim. This implies that $P = 0$,
as otherwise $P$ will be a nonzero $\fk_0$-map between $V_1$ and $V_2$.
\end{proof}

\begin{lemma}
\label{Lemma2}
Suppose that $V$ is a unitary $(\fg_0,\fk_0)$-module such that 
the $V_\theta$ are all finite dimensional. Then for any submodule $W\subset V$, 
$W^\perp$ is also a submodule, and $V = W \oplus W^\perp$.
\end{lemma}
\begin{proof}
It is only a question of proving that $V = W \oplus W^\perp$. The point
is that $V$ is in general not complete. Now $W = \oplus_ \theta W_\theta$ 
where the $W_\theta$ are
finite dimensional and mutually orthogonal, and 
$W_\theta\subset V_\theta$. Let $W_\theta'$  be the
orthogonal complement of $W_\theta$ in $V_\theta$. 
Since the isotypical subspaces of $V$
are mutually orthogonal, it is clear that $W_\theta'$  is $\perp$ to all 
$V_{\theta'}$  for $\theta'\neq\theta$. Thus
$W_\theta'\subset W^\perp$. Since this is true for 
all $\theta$, we see that $W \oplus W^\perp \supset W_\theta \oplus W_\theta' =
V_\theta$ for all $\theta$ (as the $V_\theta$ are finite dimensional). 
So $W \oplus W^\perp = V$.
\end{proof}

\begin{lemma}\label{lemma3.3} 
If $V$ is as in the previous lemma, then $V$ is the orthogonal direct sum of
irreducible submodules.
\end{lemma}

\begin{proof}
We shall first show that if $W\subset V$ is any submodule, then $W$ has
an irreducible submodule. This is a standard argument of Harish-Chandra.
Consider pairs $(W' , \theta)$ for submodules $W'\subset W$ and 
$\theta$ such that $W_\theta'\neq 0$.
Among these choose one for which $W_\theta'$  has the smallest dimension; let
$(W', \theta)$ be the corresponding pair. Let $W''$ be the cyclic 
submodule of $W'$ generated by $W_\theta'$. If $L$ is a proper submodule of 
$W''$, then we claim that $L \cap W_\theta'$  is either $0$ or equal to 
$W_\theta'$. Otherwise dim$(L_\theta)$  is positive
and strictly less than dim($W_\theta'$ ), a contradiction. 
It cannot equal $W_\theta'$, as
then $L = W''$. So $L \cap W_\theta'= 0$, hence $L \perp W_\theta'$. 
In other words all proper
submodules of $W''$ are orthogonal to $W_\theta'$, 
showing that their sum is still
proper. Let $Z$ denote this sum. Then $W'' \cap Z^\perp$ 
is an irreducible submodule of $W''$.
%\end{proof}

This proves the existence of irreducible submodules of $V$. Let ($V_i$)
be a maximal family of mutually orthogonal irreducible submodules of
$V$. If $Y := \oplus_i V_i \neq V$, then $Y^\perp$ 
will contain an irreducible submodule,
contradicting the maximality of ($V_i$). Hence the lemma.
\end{proof}

\begin{corollary}\label{cor-sec3}
Let the notation be as above.
If $V$ is finitely generated, then $V$ is an orthogonal direct sum
of finitely many irreducible submodules.
\end{corollary}

\begin{proof}
Each generator lies in a finite sum of the $V_i$. Since there are only
finitely many generators, the corollary follows.
\end{proof}

We are now ready to prove our main result for this section.

\medskip
\textit{Proof of Theorem \ref{main0}}.
Since $\fg_0$ leaves invariant $V_0$, $V_1$  
separately, $\pi_0$ can be constructed separately on the closures of $V_0$ and 
$V_1$, by
the Theorem \ref{Theorem1Wallach} in Subsec. \ref{hc-ordinary}, 
and so the full $\pi_0$ is even.
We know that $V\subset C^\omega (\pi_0)$, again by  Theorem \ref{Theorem1Wallach}.
Theorem \ref{theorem1} of Sec. \ref{intro-sec} now proves the present theorem.

\begin{remark}
In the special case of highest weight modules, the proof of 
unitarizability is simpler.
%  in the case of HW modules (highest weight).
In view of our corollary to Lemma \ref{lemma3.3},
it is enough to show, besides (U1) and (U2), only that the $V_i$ 
are highest weight modules for $\fg_0$, because 
the conditions in Cor. \ref{cor-sec3} are
automatically verified (see \cite{musson} Ch. 8).%\comment{check category O}
\end{remark}

\subsection{Construction of Harish-Chandra modules for $(\fg, \fk)$}
Apart from the highest weight
modules we have not produced any other Harish-Chandra modules
(see \cite{cfv}). In this section we
do precisely this. %Our first result is the following.
We need some preliminary remarks. 

\medskip
Let $M$  be a Harish-Chandra module for $\fg_0$
and define $V := \cU (\fg) \otimes_{\cU(\fg_0 )} M$.

\medskip
By Poincar\'e-Birkhoff-Witt theorem, if $X_1 , X_2 ,$ $\dots ,$ $X_r$ is a basis
for $\fg_1$, and $L$ is the span of all the $X_{i_1}  \dots X_{i_m}$ where 
$i_1 < . . . <i_m$ , $m \leq r$,
then $\cU (\fg) =L\,\cU (\fg_0 )$. Although $\fg_1$ is stable under 
$\ad(\fg_0 )$, 
this is not true of $L$. Let $R$ be the linear span of all monomials 
$X_{i_1} \dots X_{i_m}$ where the $i$'s
are not ordered and satisfy only $1 \leq i_1 , i_2 , \dots , i_m \leq r$, 
$m \leq r$. Then $R$ is
finite dimensional, stable under $\ad(\fg_0 )$, graded, and 
$R\,\cU (\fg_0 ) = \cU (\fg)$.
Hence
$$
\cU (\fg) \otimes_{\cU(\fg_0 )} M = R \otimes_{\cU(\fg_0 )} M.
$$
The action of $\cU(\fg)$ on $V$ is by the left on the first factor. Since $R$ 
is $\ad(\fg_0 )$-stable, the action of $\fg_0$ is the tensor product of 
the adjoint action on $R$
and the action on $M$. We recall a well known result. If $p$, $q$, $r$ are three
irreducible representations of $\fk$, we write $p < q \otimes r$ if $p$ occurs in $q \otimes r$.
Then:
$$
p < q \otimes r \quad \iff \quad r^* < q \otimes p^*
$$
where $a^*$ is the dual representation of $a$. This follows from the fact that
$p < q \otimes r$ if and only if $q \otimes r \otimes p^*$ contains the trivial representation, and
hence if and only if $r^* < q \otimes p^*$.

\begin{proposition}\label{prop1}
Let $M$ be a Harish-Chandra module for $\fg_0$. Then $V := 
\cU (\fg) \otimes_{\cU(\fg_0 )} M$
is a Harish-Chandra module for $(\fg, \fk)$.
\end{proposition}
%Proof of Proposition 1. 

\begin{proof}
We must show that for any irreducible class
$p$ of $\fk$, $dim(V_p ) < \infty$. Let $r_1 , \dots r_t$ be the irreducible classes in $R$. $M$ is
the direct sum of the $M_q$ for the various irreducible classes $q$ of $\fk$, and we
know that $dim(M_q ) < \infty$ for all $q$. Now, by our remark above, $p$ occurs
in $r \otimes q$ if and only if $q^* < r \otimes p^*$. Taking $r = r_1 , \dots , r_t$ and fixing $p$, this
gives only finitely many choices for $q$. Let $Q$ be the finite set of $q$ such
that $q^* < r_j \otimes p^*$ for some $j = 1, 2, \dots , t$. Hence
$$
V_p \subset R \otimes \oplus_{ q\in Q} M_q 
$$
showing that $dim(V_p ) < \infty$.
\end{proof}

\begin{remark} By a slight variation of the argument in the Lemma \ref{lemma3.3} 
we can show
that $V$ has subquotients which are irreducible. 
Starting with a module
$M$ for which the weight spaces are not all finite dimensional, one of the
subquotients of a finite composition series for $V$ will have this property
and so will not be a highest weight module. These modules were studied
in \cite{ev, enright} for ordinary Lie algebras and the above theory
allows us to build non highest weight Harish-Chandra modules for
Lie superalgebras. We plan to explore this further in a forthcoming paper.
\end{remark}

\section{Infinitesimal unitarity} \label{inf-unit-sec}

\subsection{Harish-Chandra homomorphism} \label{hc-homo-sec}
Let $\fgc$ be a contragredient complex Lie superalgebra (see \cite{kac}).
The Harish-Chandra homomorphism
\beq
\beta\colon Z(\fgc)\to S(\fh_\C)^W
\eeq
identifies the center
$Z(\fgc)$ of the universal enveloping
algebra with the subalgebra $I(\fh_\C)$ of $S(\fh_\C)^W$ (see \cite{ka4}):
$$
I(\fh_\C)=\{ \phi \in S(\fh_\C)^W \, | \, \phi(\lambda + t \al)=
\phi(\lambda), \,\forall\, \lambda \in \langle \al \rangle^\perp, 
\, \al \, \hbox{isotropic},\, \forall\,t\in \C\}
$$
For any $\mu \in \fh_\C^*$, 
let $\cU[\mu]$ be the subspace of $\cU(\fgc)$ given by 
$$
\cU[\mu] = \{a \in \cU(\fg_\C) \,| \, [H, a] = \mu(H)a \qquad\forall H \in \fh_\C\}. 
$$
Then $\cU [0]$ is a subalgebra, $Z(\fgc) \subset \cU [0]$,
and $(\cU [\mu])$ is a grading of $\cU(\fgc)$; 
moreover $\cU [\mu] \neq 0$ if and only if $\mu$ is in the
$\Z$-span of the roots. % (the root lattice). 
If $\gamma_1 , \dots , \gamma_t$
is an enumeration of the positive roots, 
$\Delta^+=\{\gamma_1 , \dots , \gamma_t\}$ and $(H_i )$ is a basis for $\fh_\C$, 
then elements of $\cU[0]$ are linear combinations of monomials:
\beq \label{u0el}
X^{p_1}_{-\gamma_1} \dots H_1^{c_1} \dots X_{\gamma_1}^{n_1}
\eeq
with 
$(p_1 - n_1 )\gamma_1 + \dots = 0$. It is
then clear that every term occurring in such a linear combination must
necessarily have some $p_i > 0$ except those that are just monomials in the
$H_i$ alone. So for any $u \in \cU[0]$ we have an element 
$\beta(u)\in \cU (\fh_\C)$ such that
\beq\label{eqbeta}
u \cong \beta(u)(mod \cP), \qquad \cP =\sum_{\ga > 0} \cU (\fgc)\fg_\gamma,
\quad
\gamma  \in \Delta^+
\eeq 
Let $\lambda \in \fh_\C^*$.
The action of $u$ on the Verma module $V_\lambda$ must leave the weight spaces
stable since it commutes with $\fh_\C$, and so applying it to the highest weight
vector $v_\lambda$ we see that $uv_\lambda = \beta(u)(\lambda)v_\lambda$ 
where we are identifying $\cU (\fh_\C)$
with the algebra of all polynomials on $\fh_\C^*$,
so that $\beta(u)(\lambda)$ makes sense. It
follows from this that if $u \in U (\fh_\C) \cap \cP$ then 
$u(\lambda)$ = 0 for all $\lambda$ and so $u = 0$,
i.e., $\cU (\fh_\C) \cap \cP = 0$. Hence $\beta(u)$ 
is uniquely determined by the equation
(\ref{eqbeta}).

\medskip
We extend the homomorphism $\beta : \cU [0] \lra \cU (\fh_\C)$ 
 to a linear map $\cU (\fgc) \lra \cU (\fh_\C)$ by making it 0 on 
$\cU [\mu]$ for
$\mu \neq 0$.

\subsection{Hermitian forms}\label{herm-sec}

Let $V$ be a complex super vector space. 
An  \textit{hermitian form} on $V$ 
is a complex valued sesquilinear form $(,)$ (linear
in the first, antilinear in the second argument) such that:
\beq \label{hp-def-eq}
(u,v)=(-1)^{|u||v|}\overline{(v,u)}, \quad \forall u,v \in V
\eeq
and $(u,v)=0$ for $|u|\neq|v|$, where $|u|$ denotes
the parity of an homogeneous element $u \in V$ (see \cite{vsv2} pg 111
and \cite{fg} Sec. 4). 
If $X$ is an endomorphism of
$V$, we define its \textit{adjoint} $X^*$ as 
\beq \label{adjoint-def}
(Xu,v)=(-1)^{|u||X|}(u,X^*v), 
\eeq

One can immediately verify
that (see \cite{vsv2} pg 110):
\beq\label{hermpr}
\langle u,v \rangle = \begin{cases} i(u,v) & |u|=|v|=1 \\
(u,v) & \hbox{otherwise}
\end{cases}
\eeq
is an ordinary hermitian form. If $X^\dagger$ is the adjoint
with respect to this ordinary hermitian form, we have
that $X^*=i^{|X|} X^\dagger$. 
In fact, taking $|u|=|X|=1$, $|v|=0$, we have
$(Xu,v)=-(u,X^*v)$ and
\beq\label{dagger}
(Xu,v)=-i \langle Xu,v \rangle = -i \langle u, X^\dagger v \rangle
\eeq 
A similar calculation is done if $|u|=0$, and $|X|=|v|=1$.

$(,)$ is an \textit{hermitian product} on $V$ 
if $(,)$ and $i(,)$ are 
positive definite on $V_0$ and $V_1$ respectively, i.e.
if the ordinary form $\langle \, ,\, \rangle$ is positive definite on $V$.

Let $V$ be a $\fgR$ module, $\fgR$ a real
Lie algebra, via the representation $\pi$. %Then 
$V$ (or $\pi$) is said to be {\it unitary} if there is an
hermitian product $( , )$ for $V$ such that
\beq\label{unitary-rep}
(\pi(X)u, v) = -(-1)^{|u||X|}(u, \pi(X)v)
%\quad (x,y)=0, \, |x| \neq |y|
\qquad (u, v \in V, X \in \fgR ),
\eeq
(see \cite{vsv2} pg 111). This implies:
\beq\label{unit-conds}
\pi(X)^* =\begin{cases}
-\pi(X), \qquad |X|=0 \\
+\pi(X), \qquad |X|=1
\end{cases} \qquad \qquad 
\pi(X)^\dagger =\begin{cases}
-\pi(X), \qquad |X|=0 \\
-i\pi(X), \qquad |X|=1
\end{cases}\qquad \qquad
\eeq
As one can readily check, this is equivalent to (U1), (U2) in Sec. 
\ref{intro-sec}, with (\ref{hermpr}) as hermitian product there. In fact,
while condition (U1) regards the ordinary case, 
condition (U2) is expressed for $|X|=|u|=1$ (similarly for
$|X|=|v|=1$) as:
$$
\langle e^{-i\pi/4}\pi(X)u,v \rangle =  \langle u,  e^{-i\pi/4}\pi(X)v \rangle
$$
that is:
$$
\langle \pi(X)u,v \rangle =  i\langle u,  \pi(X)v \rangle
$$
This implies the condition of unitarity to be $\pi(X)^\dagger=-i\pi(X)$, 
in agreement with (\ref{unit-conds}).

\medskip
Let $\fgR$ be a real form of contragredient complex superalgebra
$\fgc$ and $\fgR=\fkR
\oplus \fpR$ its Cartan decomposition. We assume $\fgc$ to
satisfy the {\sl equal rank condition}:
$$
\fhc \subset \fkc \subset \fgc
$$
for a fixed Cartan subalgebra $\fhc$. Assume also that
$\fkc$ has a non trivial center.
Then $\fkc$ and $\fpc$ decompose into the sum of root spaces
and the root system of $\fgc$ has admissible positive systems
and we fix $P$ one of such (see \cite{cf1}).
We can extend the antiautomorphism
$X \lra -(-1)^{|X|}X$ on $\fgR$ to an antiautomorphism of
$\cU (\fgR)$. Then, this antiautomorphism can be furtherly 
extended to a conjugate linear antiautomorphism
of $\cU(\fgc)$, that we denote by $a \lra a^\star$ and call the \textit{adjoint}.
It has the following 
properties: 
\begin{enumerate}
\item $a^{\star\star} = a$ 
\item $(ab)^\star =(-1)^{|a||b|} b^\star a^\star$ 
\item $a^\star$ is conjugate linear in $a$ 
\item $X^\star = -(-1)^{|X|}X$ forall $X \in \fgR$. 
\end{enumerate}
%(note: the sign in (2) is at the moment immaterial).
It is uniquely determined by these requirements. 
We then can extend the unitary condition 
for a representation expressed in (\ref{unitary-rep}):
$$
%\langle \pi(X)u,v \rangle =  \langle u,  \pi(X)^\star v \rangle,
(\pi(X)u,v)=(u,  \pi(X)^\star v)
\qquad X \in \cU(\fgc)
$$

\subsection{Unitary highest weight representations}
We now wish to give a criterion for an highest weight
representation of $\fgc$ to be unitary. We shall follow closely \cite{hc-iv}.

\begin{lemma}\label{lemma1}
Let $\pi_\lambda$ be a unitary highest weight representation of $\fgc$
of highest weight $\lambda$. Then 
$(-i)^{|a|}\beta(a^*a)(\la)>0$ 
for all $a \in \cU(\fgc)$.
\end{lemma}

\begin{proof}
Let $v$ be the highest weight vector. It is not restrictive to assume $v$
to be even. By definition of $\beta$ we have:
$$
(av,v)=\beta(a)(\lambda)(v,v), \qquad a \in \cU(\fgc)
$$
Hence: 
$$
0<i^{|a|}(av,av)=i^{|a|}(-1)^{|a|}(a^*av,v)=(-i)^{|a|}\beta(a^*a)(\la)(v,v)
$$
which gives our claim, since $(v,v)=\langle v,v\rangle >0$.
\end{proof}

To ease the notation let $\be_\lambda(a):=\be(a)(\lambda)$, $a \in \cU(\fgc)$.

\begin{lemma}\label{lemma2}
Assume $(-i)^{|a|}\beta_\la(a^\star a)\geq 0$ for all $a \in  \cU(\fgc)$.
Then:
\beq\label{form}
(w,z)=(-i)^{|z||w|}\beta_\la(z^\star w), \qquad w,z \in \cU(\fgc)
\eeq
defines a semipositive definite
supersymmetric sesquilinear form %positive semidefinite 
on  $\cU(\fgc)$, whose radical $\cR$ is a left ideal.
\end{lemma}

\begin{proof}
By (\ref{form}) and by the definition of $\beta$ and $\star$
we immediately have that $(,)$ is sequilinear: linear in the
first and antilinear in the second argument.  
%To see it is supersymmetric, 
Notice that 
$(w,z)=0$ if $|z| \neq |w|$. In fact, if $|z| \neq |w|$,  $|z^\star w|=1$,
hence $z^\star w \not\in \cU[0]$, which consists of even elements only
and $\beta_\la$ is zero on $\cU[\mu]$, for $\mu \neq 0$.
Moreover since $(-i)^{|a|}\beta(a^\star a)(\la) \geq 0$, we have
that $\langle a, a \rangle \geq 0$, by (\ref{hermpr}).
By a standard argument in ordinary linear algebra (see
\cite{pedersen}),
this implies that 
$\langle a, b \rangle = \overline{\langle b, a \rangle}$, $|a|=|b|$
and this concludes the first part of the proof.

\medskip
Let $\cR$ be the set of $z \in \cU(\fgc)$ with $\|z\|:=\sqrt{\langle z,
z \rangle}=0$.
If $z,z' \in \cR$, that is $\|z\|=\|z'\|=0$, by $\|z+z'\|\leq \|z\|+\|z'\|$,
we immediately have that $\cR$ is a subspace. Furthermore
if 
$b^\star=b'$, for $b,b'\in \cU(\fgc)$, we have:
$$
(w,bz)=(-i)^{|w||bz|}\beta_\la((bz)^\star w)=\pm (-i)^{|w||b||z|}
\beta_\la(z^\star b^\star w)=\pm(b'w,z)
$$
Hence $\|(w,bz)\|\leq \|b'w\|\|z\|$. If $z\in \cR$, 
i.e. $\|z\|=0$ and $w=bz$, we
get $\|bz\|=0$, hence $\cR$ is a left ideal.
\end{proof}
\begin{lemma}\label{lemma3}
Let the notation be as above. 
Assume $(-i)^{|a|}\beta_\la(a^\star a) \geq 0$ for all $a \in  \cU(\fgc)$.
Then $\cU(\fgc)/\cR$ is
a unitary representation of $\fgc$ with highest weight $\la$.
\end{lemma}

\begin{proof}
One can check right away that $(,)$ is well defined on
$\cU(\fgc)/\cR$. To prove our claim, 
we need to show $\cR=\cM_\la$ the (unique) maximal ideal
containing $\cP_\la=\sum_{\ga \in P}\cU(\fgc)\fgc_\ga+$
$\sum_{\ga \in P}\cU(\fgc)(H_\ga-\la(H_\ga)1)$.
%One sees easily that $\cP_\la \subset \cR$. 
We have $\cP_\la \subset \cR$. They are both left ideals,
so it is enough to show $X_\al \in \cR$ for $\al>0$ and
$H_\ga-\la(H_\ga)1 \in \cR$, the latter being an ordinary statement,
so true for the ordinary theory.
Notice that:
$(X_\al,v)=\be_\la(v^*X_\al)=0$ because of (\ref{eqbeta}), hence
$X_\al \in \cR$. 

\medskip
Now let $\cM'$
be a proper maximal ideal containing $\cR$. We want to
show $\cM'=\cR$. We first notice that it is stable under the $\fh_\C$
action, in fact:
$$
[H,m]=Hm-m(H-\la(H))-\la(H)m \in \cM', \qquad H\in \fh_\C, \quad m \in \cM'
$$
By a standard fact, then also $m_0$, the $\cU[0]$ component of $m$
is in $\cM'$. Then by (\ref{eqbeta}) $m_0 \equiv h$ mod ($\cP$)
and $h \equiv \beta_\la(h)$ mod($\cP_\la$),
so that $m_0 \equiv \beta_\la(h)$  mod($\cP_\la$) for some $h \in 
\cU(\fh_\C)$ ($\cP \subset \cP_\la$).
Since $\cP_\la\subset \cR
\subset \cM'$, we have $\beta_\la(h) \in\cM'$, and being a complex number,
this tells that $\beta_\la(h) =0$, otherwise $\cM'$ would not be
a proper ideal. Hence, also $\beta_\la(m_0) =\beta_\la(h) =0$. 
Now, let $z \in \cM'$. Since $X_\al^\star=c_\al X_{-\al}$ for any root vector
$X_\al$, $\al$ a root of $\fgc$ (see \cite{fg} Sec. 4), we have
$z^\star z \in \cU[0]$.
Taking $m_0=z^\star z$, for any $z\in \cM'$, this gives 
$(-i)^{|z|}\beta_\la(z^\star z)=\langle z,z \rangle =0$, so $\cM'=\cR$. 
\end{proof}

We are ready for the main result of this section.

\medskip
\textit{Proof of Theorem \ref{main1}}.
The first statement is immediate from the previous lemmas.
The second statement comes from the ordinary result in \cite{hc-iv} and easy
calculations. %\comment{R. Check} 

\section{Irreducible representations of $\rosp_\R(1|2)$}

\subsection{Introductory remarks.} 
We present here some calculations on highest
weight Harish-Chandra modules for $\fg_\C = \rosp_\C (1|2)$ and the unitary ones of 
$\fg=\rosp_\R(1|2)$.

Let $\N = \{0, 1, 2, . . .\}$. We assume that $t \not\in \N$.

The Lie superalgebra $\fgc$ consists of matrices
$$
\begin{pmatrix}
  0 &\xi & \eta\\
  \eta & a & b\\
  -\xi & c & -a
\end{pmatrix}
$$
where $\xi$, $\eta$ are complex odd variables, $a$, $b$, $c$ complex 
even variables. 
The real form $\fg$ consists of the real Lie superalgebra of matrices
$$
\begin{pmatrix}
  0 &\xi & -i\xibar\\
  -i\xibar & ia & b\\
  -\xi & -\bbar & -ia
\end{pmatrix}
$$
where the variables $\xi$, $b$ are still complex, $a$ is real, and bar denotes complex
conjugation (see \cite{ccf} Appendix A for notation).

The complex basis of $(\fg_\C)_0 = \fsl(2)$ is the standard one $H=E_{22}-E_{33}$, 
$X=E_{23}$, $Y=E_{32}$, $E_{ij}$ denoting the elementary matrices
(see \cite{vsv1} for notation).  The complex basis for the odd part $(\fg_\C)_1$ 
is $\{x, y\}$ where
$$
x = E_{13} + E_{21}, \qquad
y = E_{12} - E_{31}.
$$
For the real form, the even part has real basis $\{iH, X + Y, i(X - Y )\}$ and
the odd part has real basis
$$
x^\sim = -ix + y,\qquad
y^\sim = -x + iy.
$$

\subsection{ Verma modules for $(\fg_\C)_0$ with highest 
weight $t \not\in \N$.} We recall here the ordinary theory.
Let $W_t$ be the Verma module for $(\fg_\C)_0$
of highest weight $t$. Then $W_t$ has basis $\{v_t, v_{t-2}, \dots\}$ where 
$v_t \neq 0$,
$Xv_t =0$, $v_{t-2r} = Y_r v_t$. One knows that all the $v_{t-2r}$ are non zero,
 because of
the identity:
$$
XY^{r+1} = Y^{r+1} X + (r + 1)Y^r (H - r)
$$
in $\cU ((\fg_\C)_0 )$, established easily by induction on $r$. This shows that
\beq\label{(1)}
Xv_{t-2(r+1)} = (r + 1)(t - r)v_{t-2r}.
\eeq
Since $t \not\in \N$,
the factor $(r + 1)(t - r)$ is not zero for any integer $r \geq 0$, it
follows that if some $v_{t-2(r+1)} = 0$, then $v_{t-2r} = 0$, so that we eventually
get $v_t = 0$. That this is irreducible already is seen because of 
(\ref{(1)}). Indeed
(\ref{(1)}) shows that starting with any $v_{t-2r}$, we can reach $v_t$ 
by applying $X$
repeatedly. Thus the Verma modules $W_t$ are already irreducible.
We now want to determine when the $W_t$ are unitary. By unitary we mean
the existence of a hermitian product such that
$$
(Zu, v) = -(u, Zv)
$$
for all $Z$ in the real form of $\fsl(2)$, and $u, v \in W_t$, i.e., 
for $Z = iH$, $X +
Y$, $i(X - Y )$. The main idea is to transfer the condition for unitarity to
the complex Lie algebra $\fsl(2)$. For the Verma modules 
$W_t$, unitarity is equivalent
to assuming that the $v_{t-2r}$ are mutually orthogonal and $X^* = -Y$ or
$Y^* = -X$ or both. In fact, the condition is that $H^* = H$, $(X + Y )^* =
-(X + Y )$, $(X - Y )^* = X - Y$.

\begin{proposition} $W_t$ is unitary if and only if $t$ is real and $t < 0$.
  \end{proposition}
\begin{proof} Recall that $t \not\in \N$. Let $W_t$ be unitary. Then 
$(v_{t-2}, v_{t-2} )$ $=$
$(Yv_t, v_{t-2} ) = -(v_t,Xv_{t-2} )$. But:
$$
Xv_{t-2} = XY v_t = Y Xv_t + Hv_t = tv_t.
$$
Hence $(v_{t-2}, v_{t-2} ) = -t > 0$ if we normalize $(v_t, v_t ) = 1$ (possible). Hence
$-t > 0$.
For the converse we must, when $t < 0$, define a unique hermitian product such
that $(v_t, v_t ) = 1$ and $X^* = -Y$. The $v_{t-2r}$ are to be mutually orthogonal
and so we need to determine the $N (r) := (v_{t-2r}, v_{t-2r} )$ inductively so that
$X^* = -Y$ and all the $N (r)>0$. The requirement $X^* = -Y$ forces
the relation (by (\ref{(1)})):
$$
(v_{t-2r}, v_{t-2r} ) = (Y v_{t-2(r-1)}, v_{t-2r} ) = -r(t - r + 1)(v_{t-2(r-1)}, 
v_{t-2(r-1)} )
$$
or
$$
N (r) = -r(t - r + 1)N (r - 1), \qquad
N (1) = 1
$$
the second being the normalization $(v_t, v_t ) = 1$. We define $N (r)$ 
inductively by this and note that for $t < 0$ we have $N (r) > 0$ 
for all $r$, since
the factor $-r(t - r + 1)$ is always $> 0$ for $r \geq 1$, as $t < 0$. 
The hermitian
product is now well defined and positive definite. It is now only a question
of verifying that $X^* = -Y$. For this we need only check
$(Yv_{t-2(r-1)}, v_{t-2r} ) = -(v_{t-2(r-1)}, Xv_{t-2r} )$
as all other hermitian products needed are zero. But the 
left hand side is $N (r)$ while
the right hand side is $-r(t-r+1)(v_{t-2(r-1)}, v_{t-2(r-1)} )$, 
which is $-r(t-r+1)N (r-1)$,
and these are equal by definition of $N (r)$.
\end{proof}

\subsection{Super Verma modules for $\fg_\C$}
We report here for completeness some preliminary results.
Let $t \not\in \N$ where $\N = \{0, 1, 2,..., \}$ and let $V_t$ be the $\fg_\C$ 
module with highest weight $t$ and highest weight vector $v_t$. 
%If  $v_t\neq 0$ is a highest
%weight vector, define $v_{t-1}$ as $yv_t$; we have $v_{t-1}\neq 0$, 
%$Xv_{t-1} =0$.

\begin{lemma} 
Let $t \not\in \N$ where $\N = \{0, 1, 2,..., \}$. Then the super Verma
module $V_t$ with highest weight $t$ is irreducible.
Let $W_t$ = $\cU((\fg_\C)_0)v_t$ and  $W_{t-1}$ = $\cU((\fg_\C)_0)v_{t-1}$,
$v_{t-1}:=yv_t$.
%(similarly for $W_{t-1} $). 
Then $W_t$, $W_{t-1}$ are irreducible Verma modules
for $(\fg_\C)_0$  and $V_t$ = $W_t\oplus W_{t-1}$.
\end{lemma}

\begin{proof}
By Poincar\'e-Birkhoff-Witt theorem, 
$\cU(\fg_\C) = \cU((\fg_\C)_0)\{1, x, y, yx\}$. Hence 
$$
V_t = \cU((\fg_\C)_0)v_t+\cU((\fg_\C)_0)v_{t-1}. 
$$
If $v_{t-1} = 0$ then $xv_t = yv_t = 0$, hence, as $H = xy + yx$, we
have $Hv_t$ = $tv_t = 0$ showing that $t = 0$. Also if $Xv_{t-1} \neq 0$,
then it 
has weight $ t + 1$
which is impossible. The modules $\cU(\fg_\C)v_t$, $\cU(\fg_\C)
v_{t-1}$ are then highest weight non zero modules, 
of highest weights $t$, $t-1$. Hence, 
by our assumption that $t \not\in \N$, they
are Verma modules and irreducible. Note the sum is direct since $H$ has disjoint
spectra in the two pieces. Hence the result.
\end{proof}

\begin{corollary} Let the notation be as above.
$V_t$ has basis $\{v_t, v_{t-1}, \dots \}$ where $v_{t-r} = y^r v_t$.
\end{corollary}

\begin{proof}
Recall that $y^2 = -Y$. 
Given $t \not\in  \N$ there is only one structure of a super 
$\fg_\C$ module for $W_t\oplus W_{t-1}$,
namely the super Verma with highest weight weight $t$.
\end{proof}

\begin{lemma} \label{lemmac}
Let the notation be as above.
In $\cU (\fg_\C)$ we have
$$
xy^{2m} = y^{2m} x - my^{2m-1},\qquad
xy^{2m+1} = -y^{2m+1} x + y^{ 2m} (H - m)
$$
In particular, in $V_t$,
$$
xv_{t-m} = c_m v_{t-m+1}, \qquad
c_{2m} = -m, \qquad
c_{2m+1} = t - m.
$$
\end{lemma}

\begin{proof} Since $xy = -yx + H$ in $\cU (\fg_\C)$, we have,
by direct calculation,
$xy^2 = y^2 x - y$ and $xy^3 = -y^3 x + y^2 (H - 1)$. Hence the results are true
for $m = 1$. We use induction on $m$. We have
$$
xy^{ 2m+2 }= xyy^{ 2m+1} = (-yx + H)y^{ 2m+1} = y^{ 2m+2} x - (m + 1)y^{ 2m+1}
$$
and
$$
xy^{ 2m+3} = xyy^{ 2m+2} = (-yx + H)y^{ 2m+2} = -y^{ 2m+3 }x + y^{ 2m+2} (H - m - 1)
$$
by direct calculation. The induction is complete. The formulae for $V_t$ are
immediate consequences by applying them to $v_{t-m}=y^{t-m}v_t$.
\end{proof}

\subsection{Unitary Super Verma modules for $\fg_\C$}
The main idea now is to transform the condition for unitarity to the complex
setting. For the ordinary Verma modules $W_t$, unitarity is equivalent to $X^* = -Y$ or 
$Y^* =-X$ or both. In fact, the condition is that $H^* = H$, $(X + Y )^* = -(X +
Y )$, $(X - Y )^* = X - Y$. For supermodules we impose, following \cite{CCTV} and (U2) as in Sec.
\ref{intro-sec},
the condition $\zeta Z$ is symmetric in the hermitian product where 
$\zeta = e^{-i\pi/4}$ (see \cite{CCTV}).
As in the Verma case, we must convert this definition into a condition on
the complex basis for $\fgc$. The condition is that $\zeta x^\sim$, $\zeta y^\sim$ are symmetric
(acting on the module) where $x^\sim = -ix + y$, $y^\sim = -x + iy$. This is the
same as (see (\ref{unit-conds})):
$$
x^{{\sim}\dagger} = -ix^{\sim}, \qquad
y^{{\sim} \dagger} = -iy^{\sim}
$$
or
$$
ix^\dagger + y^\dagger = -x - iy, \qquad
-x^\dagger - iy^\dagger = ix + y.
$$
These are the same as
$$
x^\dagger = -y,
y^\dagger = -x
$$
or even just one of these relations, as the other follows by taking adjoints.
%\end{proof}
Notice that for $u,v$ even $u^\dagger=u^*$ and $(u,v)=\langle u, v \rangle$,
see Sec. \ref{herm-sec} for the notation.

\begin{theorem} \label{main-ex-thm}
$V_t$ is unitary if and only if $t$ is real and $t< 0$.
\end{theorem}

\begin{proof} Recall that $t \not\in \N$. 
Let $V_t$ be unitary. Then 
$$
xv_{t-1} = xyv_t =-yxv_t + Hv_t = tv_t
$$ 
 So 
$$
\langle v_{t-1}, v_{t-1} \rangle = \langle yv_t, v_{t-1} \rangle = 
-\langle v_t, xv_{t-1} \rangle =
-t\langle v_t, v_t \rangle
$$
We can normalize $\langle v_t, v_t \rangle = 1$ 
so that we get $\langle v_{t-1}, v_{t-1} \rangle = -t$.
Thus we must have $t < 0$.
We now prove the converse. If $t < 0$ we must define a hermitian product on
$V_t$ such that $x^\dagger = -y$.

The definition of the hermitian product goes as in the Verma case. The
formulae for $c_m$ of Lemma \ref{lemmac}
show that for $m \geq 1$, we see that $c_m < 0$ always. Let
$N (r) = \langle v_{t-r}, v_{t-r} \rangle$. The relation $x^\dagger = -y$ forces the relation
$$
\langle v_{t-r}, v_{t-r} \rangle = \langle yv_{t-r+1}, v_{t-r} \rangle = 
-c_r \langle v_{t-r+1}, v_{t-r+1} \rangle
$$
or
$$
N (r) = -c_r N (r - 1).
$$
We define $N (r)$ inductively with $N (0) = 1$. 
Then, as $-c_r > 0$, the $N(r)$ are
defined and $> 0$ for all $r$. With the orthogonality of the 
$v_{t-r}$ this defines
a hermitian product for $v_t$. To prove that $x^\dagger = -y$ 
in this hermitian product we
need only check that
$\langle xy v_{t-r+1}, v_{t-r} \rangle = -\langle v_{t-r+1}, xv_{t-r} \rangle$.
The left side in $N (r)$ while the right side is $-c_r N (r - 1)$ 
and so we are
done.
\end{proof}

\begin{remark}
We observe that the necessary conditions of Theorem \ref{main1} for $V$, chosen
as in Theorem \ref{main-ex-thm}, to be unitary is satisfied, since we
have only a non compact even root $\al$ and $\lambda(H_\al)=\lambda(H)=t<0$.
\end{remark}

{\sl Proof of Theorem \ref{main-ex}}. The first statement is Theorem 
\ref{main-ex-thm}.
The second statement is an immediate consequence of the first statement 
and Theorem \ref{main0}.

\end{document}